\documentclass[11pt]{amsart}

\usepackage{amsmath,amssymb,mathtools}
\usepackage{microtype}
\usepackage[colorlinks=true,linkcolor=blue,citecolor=blue,urlcolor=blue]{hyperref}

\newtheorem{theorem}{Theorem}[section]
\newtheorem{proposition}[theorem]{Proposition}
\newtheorem{lemma}[theorem]{Lemma}

\theoremstyle{definition}

\newcommand{\Ant}{\mathcal N}
\newcommand{\Layer}{\mathcal A}
\newcommand{\HH}{\mathbb H}

\newcommand{\C}{\mathbb C}
\newcommand{\PF}{\mathrm{PF}_{\infty}}
\newcommand{\lc}{\operatorname{lc}}

\title[Real stability with a two-element factor]
{Real stability of layer-refined antichain polynomials for three-chain products with a two-element factor}
\author{Weiqi Jiang}
\address{Institute of Theoretical Physics, Chinese Academy of Sciences,
No. 55 Zhong Guan Cun East Road, Haidian District, Beijing 100190,
P. R. China}
\email{jiangweiqi@itp.ac.cn}
\date{September 8, 2026}
\hypersetup{
  pdflang={en-US},
  pdftitle={Real stability of layer-refined antichain polynomials for three-chain products with a two-element factor},
  pdfauthor={Weiqi Jiang},
  pdfsubject={Antichain polynomials of products of three chains},
  pdfkeywords={antichain polynomial, real stability, product poset, Jacobi polynomial, gamma-positivity}
}

\begin{document}

\begin{abstract}
For all positive integers \(n,k\), we prove that the layer-refined
antichain polynomial of the product poset \([2]\times[n]\times[k]\) is
real stable.  Jacobi-polynomial interlacing further shows that its diagonal
specialization, the ordinary antichain polynomial of the same poset, has only
simple, strictly negative zeros.  For the special family
\([2]\times[m]\times[m+1]\), explicit reciprocal identities give
palindromicity; reciprocal pairing of the simple negative zeros then shows
that every coefficient in the gamma expansion is strictly positive.  Thus we
prove Conjecture~4.3 of Ding and Dong and resolve all parts of their
Conjecture~4.5, while strengthening its stated gamma-positivity consequence.
The enumerative input is an explicit first-crossing reflection for two lattice
paths, specialized from work of Krattenthaler and Sulanke.
\end{abstract}

\maketitle

\section{Introduction}

For a nonnegative integer \(r\), write \([r]=\{1,\ldots,r\}\), with
\([0]=\varnothing\).  Products of
chains are ordered coordinatewise.  An \emph{antichain} is a subset whose
elements are pairwise incomparable.  If \(P\) is a finite poset, its
\emph{antichain polynomial} is
\[
  \Ant_P(x)=\sum_{S\subseteq P\,\text{an antichain}}x^{|S|}.
\]
Ding and Dong conjectured that
\(\Ant_{[2]\times[n]\times[k]}(x)\) is real-rooted and that, when
\(k=n+1\), it is palindromic with the stated gamma-positive consequence
\cite[Conjectures~4.3 and~4.5]{DingDong}.  See Dong and Weng
\cite{DongWeng} for the representation-theoretic background.  The results
below settle both conjectures and, more strongly, show that every gamma
coefficient in the latter family is strictly positive.

Let
\[
  L_i=\{i\}\times[n]\times[k]\qquad(i=1,2)
\]
be the two layers, and define
\[
  \Layer_{n,k}(u,v)
  =\sum_{S\,\text{an antichain}}
       u^{|S\cap L_1|}v^{|S\cap L_2|}.
\]
Following the standard half-plane terminology \cite{Branden}, a polynomial
with real coefficients is \emph{real stable} if it does not vanish when all
of its variables lie in
\(\HH=\{z\in\C:\operatorname{Im}z>0\}\).  The two layer weights retain
information that is lost under diagonal specialization.  Real stability
controls their simultaneous complex variation in the upper half-plane,
whereas Conjecture~4.3 concerns only the one-variable specialization
\(u=v=x\).  Thus the layer-refined statement below is
a strengthening of diagonal real-rootedness.

\begin{theorem}\label{thm:main-stability}
For all positive integers \(n,k\), the polynomial
\(\Layer_{n,k}(u,v)\) is real stable.
\end{theorem}

The diagonal identity
\(\Layer_{n,k}(x,x)=\Ant_{[2]\times[n]\times[k]}(x)\) now gives
Conjecture~4.3.  The strict interlacing input also yields the sharper
univariate statement below.

\begin{theorem}\label{thm:simple-negative}
For all positive integers \(n,k\), every zero of
\(\Ant_{[2]\times[n]\times[k]}(x)\) is simple and strictly negative.
Consequently its coefficient sequence, extended by zeros, is a
P\'olya-frequency sequence of order infinity; the finite coefficient sequence
is strictly log-concave and unimodal.
\end{theorem}

\begin{theorem}\label{thm:gamma-intro}
For every positive integer \(m\), the polynomial
\(\Ant_{[2]\times[m]\times[m+1]}(x)\) has degree \(2m\), is
palindromic, has only simple strictly negative zeros, and has strictly
positive gamma coefficients.
\end{theorem}

Permuting the factors shows that Theorem~\ref{thm:simple-negative} applies
to every three-dimensional box with a side of length \(2\); a side of
length \(1\) is covered by the Jacobi formula used below.  Thus a side of
length at most \(2\) is a sufficient condition.  Ding and Dong's
Example~3.6(a) is
\[
\Ant_{[3]\times[3]\times[3]}(x)
=1+27x+162x^2+350x^3+310x^4+114x^5+15x^6+x^7,
\]
which they note in Remark~4.4 is not real-rooted \cite{DingDong}.  Hence no
corresponding guarantee holds for all boxes whose three sides have length at
least \(3\); the example does not classify those boxes.

The proof begins with the ordinary-weight specialization of Krattenthaler and
Sulanke's two-path turn enumeration \cite{KrattenthalerSulanke}; we give the
reflection and its inverse explicitly.  The two univariate polynomials that
result are related to Jacobi polynomials, and the theorem of Driver, Jordaan,
and Mbuyi \cite{DriverJordaanMbuyi} places their zeros in strict alternation.
Finally, a positive-residue partial fraction expansion proves stability by
mapping the upper half-plane to the lower half-plane.

\section{Fixed-layer enumeration}

We first count antichains having prescribed cardinalities in the two
layers.  The details of the reflection are included both to fix the weak
inequalities in the product order and to cover all boundary cases.

Write the two parts of an antichain candidate as
\[
 A=\{(1,a_i,j_i):1\leq i\leq p\},\qquad
 B=\{(2,b_t,\ell_t):1\leq t\leq q\},
\]
where
\[
 a_1<\cdots<a_p,\quad j_1>\cdots>j_p,
 \qquad
 b_1<\cdots<b_q,\quad \ell_1>\cdots>\ell_q.
\]
Put \(c_i=k+1-j_i\) and \(d_t=k+1-\ell_t\).  Both rows in each of
\[
 \begin{pmatrix}a_1&\cdots&a_p\\c_1&\cdots&c_p\end{pmatrix},
 \qquad
 \begin{pmatrix}b_1&\cdots&b_q\\d_1&\cdots&d_q\end{pmatrix}
\]
are strictly increasing, and conversely every such pair of arrays encodes
two antichains in the grid layers.  The only possible cross-layer
comparison is
\[
 (1,a_i,j_i)\leq(2,b_t,\ell_t)
 \quad\Longleftrightarrow\quad
 a_i\leq b_t\ \text{ and }\ c_i\geq d_t.                 \tag{2.1}\label{2.1}
\]
Thus the union is an antichain exactly when no pair \((i,t)\) satisfies
the two inequalities in \eqref{2.1}.

Throughout the paper,
\[
 \binom ab=0\qquad\text{if }b<0\text{ or }b>a,
\]
where all upper arguments under consideration are nonnegative integers.
For a finite set \(E\), the notation \(\binom{E}{r}\) denotes its set of
\(r\)-element subsets.

\begin{theorem}[Fixed-layer determinant]\label{thm:determinant}
For nonnegative integers \(p,q\), let \(C_{p,q}(n,k)\) be the number of
antichains having \(p\) elements in \(L_1\) and \(q\) elements in \(L_2\).
Then
\[
 C_{p,q}(n,k)=
 \det\begin{pmatrix}
 \binom np\binom kp&
 \binom{n+1}{p+1}\binom{k-1}{p-1}\\[2mm]
 \binom{n-1}{q-1}\binom{k+1}{q+1}&
 \binom nq\binom kq
 \end{pmatrix}.                                         \tag{2.2}\label{2.2}
\]
\end{theorem}

\begin{proof}
The product of the diagonal entries counts all independently chosen pairs
of grid antichains.  We construct a bijection from the bad pairs, those
for which \eqref{2.1} holds at least once, to
\[
 \binom{[n+1]}{p+1}\times\binom{[k-1]}{p-1}
 \times\binom{[n-1]}{q-1}\times\binom{[k+1]}{q+1}.       \tag{2.3}\label{2.3}
\]
This is the ordinary-counting specialization, with the weight parameter
set to \(1\), of the two-path construction in
\cite[Eq.~(1), equivalently Eq.~(12), and Sec.~4]
{KrattenthalerSulanke}.  We spell out the specialization.

Complement and reverse the rows:
\[
\begin{aligned}
 X&=(n+1-a_p<\cdots<n+1-a_1),\\
 Y&=(k+1-c_p<\cdots<k+1-c_1),\\
 U&=(n+1-b_q<\cdots<n+1-b_1),\\
 V&=(k+1-d_q<\cdots<k+1-d_1).
\end{aligned}                                            \tag{2.4}\label{2.4}
\]
Write their entries as \(x_h,y_h,u_s,v_s\), respectively.  The rows
\(X,Y\) have length \(p\), the rows \(U,V\) have length \(q\), and the
alphabets are \([n],[k],[n],[k]\).  Under \eqref{2.4}, condition
\eqref{2.1} becomes
\[
 x_h\geq u_s\quad\text{and}\quad y_h\leq v_s             \tag{2.5}\label{2.5}
\]
for some \((h,s)\), the intersection criterion in the turn-array model.

Assume first that \(p,q\geq1\) and the pair is bad.  Select the
coordinatewise first witness
\[
\begin{split}
 I&=\min\{h:\eqref{2.5}\text{ holds for some }s\},\\
 J&=\min\{s:\eqref{2.5}\text{ holds for some }h\}.
\end{split}                                               \tag{2.6}\label{2.6}
\]
The pair \((I,J)\) is itself a witness.  Indeed, witnesses \((I,s_0)\)
and \((h_0,J)\), together with \(J\leq s_0\), \(I\leq h_0\), give
\[
 x_I\geq u_{s_0}\geq u_J,\qquad
 y_I\leq y_{h_0}\leq v_J.
\]
Set
\[
 \alpha=u_J,\qquad \delta=y_I.                           \tag{2.7}\label{2.7}
\]
If \(h<I\), then \(y_h<\delta\leq v_J\), so minimality of \(I\)
forces \(x_h<\alpha\).  Similarly, if \(s<J\), then
\(u_s<\alpha\leq x_I\), so minimality of \(J\) forces
\(v_s<\delta\).

For the initial blocks define
\[
\begin{array}{ll}
 \widehat x_t=\alpha-x_{I-t},&
 \widehat y_t=\delta-y_{I-t}\quad(1\leq t<I),\\
 \widehat u_t=\alpha-u_{J-t},&
 \widehat v_t=\delta-v_{J-t}\quad(1\leq t<J).
\end{array}                                               \tag{2.8}\label{2.8}
\]
Reflect these blocks, move the crossing coordinates into the longer rows,
and shift the tails by one:
\[
\begin{aligned}
 X^+&=(\widehat x_1,\ldots,\widehat x_{I-1},
          \alpha,x_I+1,\ldots,x_p+1),\\
 Y^-&=(\widehat y_1,\ldots,\widehat y_{I-1},
          y_{I+1}-1,\ldots,y_p-1),\\
 U^-&=(\widehat u_1,\ldots,\widehat u_{J-1},
          u_{J+1}-1,\ldots,u_q-1),\\
 V^+&=(\widehat v_1,\ldots,\widehat v_{J-1},
          \delta,v_J+1,\ldots,v_q+1).
\end{aligned}                                            \tag{2.9}\label{2.9}
\]
Empty blocks are omitted.  Each reflected block is positive and strictly
increasing.  At the splices, whenever the indicated entries exist,
\[
\begin{gathered}
 \widehat x_{I-1}<\alpha<x_I+1,
 \qquad
 \widehat y_{I-1}\leq\delta-1<y_{I+1}-1,\\
 \widehat u_{J-1}\leq\alpha-1<u_{J+1}-1,
 \qquad
 \widehat v_{J-1}<\delta<v_J+1.
\end{gathered}                                            \tag{2.10}\label{2.10}
\]
Because \(\delta=y_I\leq y_p\) and \(\alpha=u_J\leq u_q\), the endpoint
bounds are explicitly
\[
\begin{aligned}
 \max X^+&=x_p+1\leq n+1,\\
 \max Y^-&\leq\max\{\delta-y_1,y_p-1\}\leq k-1
   &&\text{if }Y^-\ne\varnothing,\\
 \max U^-&\leq\max\{\alpha-u_1,u_q-1\}\leq n-1
   &&\text{if }U^-\ne\varnothing,\\
 \max V^+&=v_q+1\leq k+1.
\end{aligned}
\]
No maximum is taken for an empty row.  The rows
\(X^+,Y^-,U^-,V^+\) have respective lengths
\(p+1,p-1,q-1,q+1\); together with positivity, the displayed bounds place
them in the respective alphabets \([n+1],[k-1],[n-1],[k+1]\).  Thus
\eqref{2.9} belongs to \eqref{2.3}.

We now give the inverse.  Start with a quadruple
\((X^+,Y^-,U^-,V^+)\) in \eqref{2.3}, append the forced sentinels
\[
 y^-_p=k,\qquad u^-_q=n,                                  \tag{2.11}\label{2.11}
\]
and search only over \(1\leq r\leq p\), \(1\leq s\leq q\).  Define
\[
\begin{split}
 I&=\min\{r:x^+_r\leq u^-_s\text{ and }y^-_r\geq v^+_s
               \text{ for some }s\},\\
 J&=\min\{s:x^+_r\leq u^-_s\text{ and }y^-_r\geq v^+_s
               \text{ for some }r\}.
\end{split}                                               \tag{2.12}\label{2.12}
\]
These sets are nonempty: the \(p\)-th entry of a \((p+1)\)-subset of
\([n+1]\) is at most \(n=u^-_q\), while the \(q\)-th entry of a
\((q+1)\)-subset of \([k+1]\) is at most \(k=y^-_p\).  The monotonicity
argument used after \eqref{2.6}, with both inequalities reversed, shows
that \((I,J)\) itself satisfies the two inequalities in \eqref{2.12}.
Put
\[
 \alpha'=x^+_I,\qquad \delta'=v^+_J,                      \tag{2.13}\label{2.13}
\]
and, for the relevant ranges, set
\[
\begin{array}{ll}
 \bar x_t=\alpha'-x^+_{I-t},&
 \bar y_t=\delta'-y^-_{I-t}\quad(1\leq t<I),\\
 \bar u_t=\alpha'-u^-_{J-t},&
 \bar v_t=\delta'-v^+_{J-t}\quad(1\leq t<J).
\end{array}                                               \tag{2.14}\label{2.14}
\]
Reconstruct
\[
\begin{aligned}
 X&=(\bar x_1,\ldots,\bar x_{I-1},
        x^+_{I+1}-1,\ldots,x^+_{p+1}-1),\\
 Y&=(\bar y_1,\ldots,\bar y_{I-1},\delta',
        y^-_I+1,\ldots,y^-_{p-1}+1),\\
 U&=(\bar u_1,\ldots,\bar u_{J-1},\alpha',
        u^-_J+1,\ldots,u^-_{q-1}+1),\\
 V&=(\bar v_1,\ldots,\bar v_{J-1},
        v^+_{J+1}-1,\ldots,v^+_{q+1}-1).
\end{aligned}                                            \tag{2.15}\label{2.15}
\]
For \(r<I\), failure of \((r,J)\) in \eqref{2.12}, together with
\(x^+_r<\alpha'\leq u^-_J\), gives \(y^-_r<\delta'\).  For \(s<J\),
failure of \((I,s)\), together with
\(v^+_s<\delta'\leq y^-_I\), gives \(u^-_s<\alpha'\).  Hence the
reflected blocks in \eqref{2.15} are positive and strictly increasing,
the splices are strict, and the witness bounds
\(\alpha'\leq n\), \(\delta'\leq k\) put all four rows in their original
alphabets.

It remains only to check that the same first crossing is recovered.  For
a quadruple obtained from \eqref{2.9}, the pair \((I,J)\) satisfies
\eqref{2.12}, using
a sentinel if \(I=p\) or \(J=q\).  If \(h<I,s\geq J\), then
\(y^-_h<\delta\leq v^+_s\); if \(h\geq I,s<J\), then
\(x^+_h\geq\alpha>u^-_s\).  In the remaining region \(h<I,s<J\), an
inverse witness is, by \eqref{2.8}, equivalent to the original witness
\((I-h,J-s)\), contrary to \eqref{2.6}.  Thus the inverse chooses the
same indices, and \eqref{2.14}--\eqref{2.15} cancel
\eqref{2.8}--\eqref{2.9} term by term.

Conversely, the quadruple reconstructed in \eqref{2.15} satisfies
\[
 x_I=x^+_{I+1}-1\geq\alpha'=u_J,\qquad
 y_I=\delta'\leq v^+_{J+1}-1=v_J,
\]
so \((I,J)\) is a forward witness.  There is no forward witness with
\(h<I,s\geq J\), because then \(x_h<u_J\leq u_s\), and none with
\(h\geq I,s<J\), because then \(y_h>v_s\).  In the remaining region
\(h<I,s<J\), a forward witness would be equivalent through
\eqref{2.14} to an earlier inverse witness.  Therefore the two maps are
mutually inverse.  The number of bad pairs is the cardinality of
\eqref{2.3}, which is precisely the product of the two off-diagonal
entries in \eqref{2.2}.

If \(p=0\) or \(q=0\), no bad pair exists and the off-diagonal product
vanishes because it contains, respectively,
\(\binom{k-1}{-1}\) or \(\binom{n-1}{-1}\).  Empty reflected rows are
allowed, and \([0]\) has one empty subset.  Our binomial convention also
covers \(n=1\), \(k=1\), \(p>\min(n,k)\), and \(q>\min(n,k)\).  Finally,
if a target set in \eqref{2.3} is empty, the forward map shows that no bad
input pair exists.  This proves all boundary cases and the determinant.
\end{proof}

\section{Layer-refined closed form}

Define the symmetric polynomials \(F_{n,k}\) and \(T_{n,k}\), respectively,
by
\[
 F_{n,k}(x)=\sum_{r\geq0}\binom nr\binom kr x^r
 ={}_2F_1(-n,-k;1;x)                                     \tag{3.1}\label{3.1}
\]
\[
 T_{n,k}(x)=\sum_{s\geq0}
 \frac{\binom{n-1}s\binom{k-1}s}{\binom{s+2}2}x^s
 ={}_2F_1(1-n,1-k;3;x).                                  \tag{3.2}\label{3.2}
\]
The coefficient identity behind the second equality in \eqref{3.2} is
\[
 \frac{(1-n)_s(1-k)_s}{(3)_s s!}
 =\frac{\binom{n-1}s\binom{k-1}s}{\binom{s+2}2}.
\]

For \(s\geq0\), with both sides zero outside their natural ranges,
\[
 \binom{n+1}{s+2}
 =\binom{n+1}{2}\frac{\binom{n-1}s}{\binom{s+2}2}.        \tag{3.3}\label{3.3}
\]
Indeed, both sides equal
\((n+1)!/((s+2)!(n-1-s)!)\).  Taking \(s=p-1\) and then exchanging
\(n,k\) gives
\[
\begin{aligned}
 \sum_p\binom{n+1}{p+1}\binom{k-1}{p-1}u^p
 &=\binom{n+1}{2}uT_{n,k}(u),\\
 \sum_q\binom{n-1}{q-1}\binom{k+1}{q+1}v^q
 &=\binom{k+1}{2}vT_{n,k}(v).
\end{aligned}                                             \tag{3.4}\label{3.4}
\]

\begin{proposition}\label{prop:closed-form}
For positive integers \(n,k\),
\[
\boxed{\begin{aligned}
 \Layer_{n,k}(u,v)
 &=F_{n,k}(u)F_{n,k}(v)-CuvT_{n,k}(u)T_{n,k}(v),\\
 C&=\binom{n+1}{2}\binom{k+1}{2}.
\end{aligned}}                                            \tag{3.5}\label{3.5}
\]
Consequently,
\[
 \Ant_{[2]\times[n]\times[k]}(x)
 =F_{n,k}(x)^2-Cx^2T_{n,k}(x)^2.                         \tag{3.6}\label{3.6}
\]
\end{proposition}

\begin{proof}
Multiply the two diagonal generating functions in the determinant of
Theorem~\ref{thm:determinant}, subtract the product of its two
off-diagonal generating functions, and use \eqref{3.4}.  This gives
\eqref{3.5}.  Setting \(u=v=x\) yields both the combinatorial diagonal
identity
\(\Layer_{n,k}(x,x)=\Ant_{[2]\times[n]\times[k]}(x)\) and
formula \eqref{3.6}.
\end{proof}

\section{Jacobi interlacing and real stability}

Put
\[
 m=\min(n,k),\qquad M=\max(n,k),\qquad
 \beta=M-m,\qquad z=\frac{1+x}{1-x}.                    \tag{4.1}\label{4.1}
\]
Since \(F_{n,k}\) and \(T_{n,k}\) are symmetric in \(n,k\), we may take
\(n=m\) and \(k=M\) in the derivation.
We use the normalization
\[
 P_r^{(\alpha,\beta)}(z)
 =\frac{(\alpha+1)_r}{r!}
 {}_2F_1\!\left(-r,r+\alpha+\beta+1;
 \alpha+1;\frac{1-z}{2}\right)                         \tag{4.2}\label{4.2}
\]
for Jacobi polynomials.  Pfaff's transformation and \eqref{4.2}, in the
forms recorded in \cite[Secs.~15.8 and 18.5]{DLMF}, give
\[
 F_{n,k}(x)=(1-x)^mP_m^{(0,\beta)}(z)                    \tag{4.3}\label{4.3}
\]
and
\[
 T_{n,k}(x)=\frac{2}{m(m+1)}(1-x)^{m-1}
 P_{m-1}^{(2,\beta)}(z).                                \tag{4.4}\label{4.4}
\]
For example, \eqref{4.3} follows by applying
\[
 {}_2F_1(a,b;c;x)
 =(1-x)^{-a}{}_2F_1\!\left(a,c-b;c;\frac{x}{x-1}\right)
\]
with \((a,b,c)=(-m,-M,1)\); equation \eqref{4.4} uses
\((a,b,c)=(1-m,1-M,3)\) and
\((3)_{m-1}/(m-1)!=m(m+1)/2\).

The variable \(z\) is defined only for \(x\ne1\), so the preceding
derivation initially proves \eqref{4.3}--\eqref{4.4} only there.  For every
Jacobi polynomial \(P_r\), however,
\((1-x)^rP_r((1+x)/(1-x))\) is a polynomial in \(x\).  Since the formulas
hold for \(x\ne1\), they hold identically, and hence also at \(x=1\).

We use the Jacobi interlacing theorem of Driver, Jordaan, and Mbuyi
\cite[Theorem~2.3]{DriverJordaanMbuyi}.  For \(\alpha,\beta>-1\) and
\(0\leq t,s\leq2\), it strictly interlaces the zeros of
\(P_r^{(\alpha,\beta)}\) with those of
\(P_{r-1}^{(\alpha+t,\beta+s)}\).  In our case, take
\[
 r=m,\qquad (\alpha,\beta)=(0,M-m),\qquad (t,s)=(2,0).
\]
For \(m\geq2\), let
\(\xi_1<\cdots<\xi_m\) and \(\eta_1<\cdots<\eta_{m-1}\) denote the
zeros of \(P_m^{(0,\beta)}\) and \(P_{m-1}^{(2,\beta)}\), respectively.
Since \(\beta=M-m\geq0\), all hypotheses are satisfied, and
\[
 -1<\xi_1<\eta_1<\xi_2<\cdots<\eta_{m-1}<\xi_m<1.        \tag{4.5}\label{4.5}
\]

Set
\[
 G(x)=xT_{n,k}(x).
\]
Since \(x/(1-x)=(z-1)/2\), equation \eqref{4.4} also gives
\[
 G(x)=\frac{(1-x)^m}{m(m+1)}
      (z-1)P_{m-1}^{(2,\beta)}(z).                       \tag{4.6}\label{4.6}
\]
The inverse M\"obius transformation is
\[
 x=\frac{z-1}{z+1},                                      \tag{4.7}\label{4.7}
\]
which is strictly increasing from \((-1,1]\) onto
\(({-}\infty,0]\).  For \(m\geq2\), define
\[
 \rho_i=\frac{\xi_i-1}{\xi_i+1}\quad(1\leq i\leq m),
 \qquad
 r_i=\frac{\eta_i-1}{\eta_i+1}\quad(1\leq i<m),
 \qquad r_m=0.
\]
Equations \eqref{4.3}--\eqref{4.6} give the strict interlacing order
\[
 \rho_1<r_1<\rho_2<\cdots<r_{m-1}<\rho_m<r_m=0.           \tag{4.8}\label{4.8}
\]
These are all the zeros, because \(F=F_{n,k}\) and \(G\) both have degree
\(m\) and positive leading coefficients:
\[
 [x^m]F=\binom nm\binom km>0,
 \qquad
 [x^m]G=
 \frac{\binom{n-1}{m-1}\binom{k-1}{m-1}}
      {\binom{m+1}{2}}>0.                                \tag{4.9}\label{4.9}
\]
When \(m=1\), the second Jacobi polynomial in \eqref{4.4} is
\(P_0^{(2,\beta)}=1\), so it has no zeros.  Directly,
\(F(x)=1+Mx\) and \(G(x)=x\), with respective zeros \(-1/M\) and \(0\).
Thus \eqref{4.8}, interpreted as \(-1/M<0\), holds for every positive
\(n,k\).  In particular, \(F\) and \(G\) have simple real zeros and no
common zero.

\begin{lemma}[Positive residues]\label{lem:residues}
There are positive real numbers \(c,a_1,\ldots,a_m\) such that
\[
 \frac{F(\zeta)}{G(\zeta)}
 =c+\sum_{j=1}^m\frac{a_j}{\zeta-r_j},
 \qquad
 c=\frac{[x^m]F}{[x^m]G},
 \qquad
 a_j=\frac{F(r_j)}{G'(r_j)}.                             \tag{4.10}\label{4.10}
\]
Consequently, \(F/G\) maps \(\HH\) strictly into the lower half-plane.
\end{lemma}

\begin{proof}
The zeros of \(G\) are simple, so partial fractions give \eqref{4.10}.
At the zero \(r_j\) of \(G\), the \(m-j\) zeros
\(\rho_{j+1},\ldots,\rho_m\) of \(F\) lie to its right.  The \(m-j\)
zeros \(r_{j+1},\ldots,r_m\) of \(G\) likewise lie to the right of
\(r_j\).  Since both polynomials have positive leading coefficients,
\(F(r_j)\) and \(G'(r_j)\) therefore have the same sign
\((-1)^{m-j}\), and hence \(a_j>0\).  The appended zero causes no exception,
since \(F(0)=G'(0)=1\).  For \(\zeta\in\HH\),
\[
 \operatorname{Im}\frac{F(\zeta)}{G(\zeta)}
 =-\operatorname{Im}(\zeta)
   \sum_{j=1}^m\frac{a_j}{|\zeta-r_j|^2}<0.              \tag{4.11}\label{4.11}
\]
There is no pole in \(\HH\), because all zeros of \(G\) are real.
\end{proof}

\begin{proof}[Proof of Theorem~\ref{thm:main-stability}]
Suppose that \(u,v\in\HH\) and \(\Layer_{n,k}(u,v)=0\).  Divide
\eqref{3.5} by the nonzero number \(G(u)G(v)\).  Since
\(G(x)=xT_{n,k}(x)\), we obtain
\[
 \frac{F(u)}{G(u)}\frac{F(v)}{G(v)}=C>0.                 \tag{4.12}\label{4.12}
\]
By Lemma~\ref{lem:residues}, both factors on the left lie strictly in the
lower half-plane.  Choose their arguments in \((-\pi,0)\).  The argument
of their product lies in \((-2\pi,0)\), which contains no multiple of
\(2\pi\), so the product cannot be a positive real number.  This
contradiction proves real stability.
\end{proof}

\section{Simple negative roots}

The same rational function proves more than diagonal real-rootedness.  Put
\[
 \alpha=\sqrt C=\frac12\sqrt{m(m+1)M(M+1)},
 \qquad
 Q_\pm(x)=F(x)\pm\alpha G(x).                            \tag{5.1}\label{5.1}
\]
Then
\[
 \Ant_{[2]\times[n]\times[k]}(x)=Q_+(x)Q_-(x).          \tag{5.2}\label{5.2}
\]
Let \(h=F/G\).  Differentiating \eqref{4.10} on the real line gives
\[
 h'(t)=-\sum_{j=1}^m\frac{a_j}{(t-r_j)^2}<0              \tag{5.3}\label{5.3}
\]
away from the poles.  Thus on every bounded interval
\((r_j,r_{j+1})\), the function \(h\) decreases from \(+\infty\) to
\(-\infty\).  On the outer intervals its ranges are
\[
 h(({-}\infty,r_1))=({-}\infty,c),
 \qquad
 h((r_m,\infty))=(c,\infty),                             \tag{5.4}\label{5.4}
\]
where \(r_m=0\), and a calculation from \eqref{4.9} gives
\[
 c=\frac{M(m+1)}2,\qquad c\geq\alpha,
 \qquad c=\alpha\ \Longleftrightarrow\ M=m.             \tag{5.5}\label{5.5}
\]

A root of \(Q_\pm\) is not a pole of \(h\), because \(F\) and \(G\)
have no common zero; it is therefore a solution of \(h=\mp\alpha\).
On each of the \(m-1\) bounded pole intervals, each target value is attained
exactly once.  Since \(-\alpha<c\), the factor \(Q_+\) has one further root
in \(({-}\infty,r_1)\) and no root in the right outer interval
\((r_m,\infty)\).  If \(M>m\), then \(\alpha<c\), so \(Q_-\) likewise has
one further root in \(({-}\infty,r_1)\) and none in \((r_m,\infty)\).
If \(M=m\), then \(\alpha=c\), and \eqref{5.4} shows that \(Q_-\) has no
finite root in either outer interval; it has only the \(m-1\) bounded roots.
There is exactly one degree drop in that case.  Indeed, from \eqref{4.10},
\[
 h(t)-c=\frac{\sum_ja_j}{t}+O(t^{-2})
 \qquad(t\longrightarrow\infty),
\]
and \(G(t)=([x^m]G)t^m+O(t^{m-1})\), where \([x^m]G>0\).  Therefore
\[
 Q_-(t)=([x^m]G)\Bigl(\sum_ja_j\Bigr)t^{m-1}+O(t^{m-2}),
 \qquad
 \lc(Q_-)=([x^m]G)\sum_ja_j>0.
\]
Thus \(Q_-\) has degree exactly \(m-1\).
The formula includes \(m=M=1\), when \(Q_-=1\).  These real roots exhaust
the factors: \(Q_+\) has degree \(m\); \(Q_-\) has degree \(m\) when
\(M>m\) and degree \(m-1\) when \(M=m\).

At a finite root \(q\) of either factor,
\[
 Q_\pm'(q)=G(q)h'(q)\neq0,                               \tag{5.6}\label{5.6}
\]
so every root of each factor is simple.  The two factors have no common
root: such a root would satisfy both \(F=0\) and \(G=0\), contradicting
\eqref{4.8}.  Hence their product \eqref{5.2} is square-free and all of
its roots are real.

Let \(w\) be the width of \([2]\times[n]\times[k]\), and write
\[
 \Ant_{[2]\times[n]\times[k]}(x)=\sum_{j=0}^w b_jx^j.
\]
A maximum antichain of size \(w\) has an antichain subset of every size
from \(0\) to \(w\).  Therefore \(b_j>0\) for every \(j\).  The polynomial
has neither a positive zero nor the zero \(0\), and all its zeros are real;
they are consequently strictly negative.  This completes the proof of the
root assertion in Theorem~\ref{thm:simple-negative}.

Newton's inequalities now give, for \(1\leq j<w\),
\[
 b_j^2\geq b_{j-1}b_{j+1}
 \frac{(j+1)(w-j+1)}{j(w-j)}
 >b_{j-1}b_{j+1}.                                       \tag{5.7}\label{5.7}
\]
Thus the coefficient sequence is strictly log-concave and unimodal.  The
finite Aissen--Schoenberg--Whitney characterization
\cite{AissenSchoenbergWhitney} says that a nonnegative finite coefficient
sequence, extended by zeros, is \(\PF\) precisely when its generating
polynomial has only real nonpositive zeros.  Hence the coefficient sequence
\((b_0,\ldots,b_w)\), extended by zeros, is \(\PF\), proving the remaining
assertions of
Theorem~\ref{thm:simple-negative}.  These consequences concern the diagonal
specialization only and do not imply a multivariate negative-dependence
property for \(\Layer_{n,k}\).

Finally, permuting the three chain factors proves the same simple-negative
root statement whenever a side has length \(2\).  If a side has length
\(1\), the antichain polynomial is \(F_{n,k}\), whose simple strictly
negative zeros follow from \eqref{4.3}, \eqref{4.5}, and the \(m=1\)
calculation.  This proves the side-length extension stated in the
introduction.

\section{Reciprocity and gamma-positivity}

Fix \(m\geq1\) and specialize to \((n,k)=(m,m+1)\).  In this section let
\[
\begin{gathered}
 F=F_{m,m+1},\qquad T=T_{m,m+1},\qquad G=xT,\\
 d=\sqrt{m(m+2)},\qquad
 s=\sqrt{\binom{m+1}{2}\binom{m+2}{2}}
   =\frac{m+1}{2}d,\\
 Q_\pm=F\pm sG,
 \qquad \rho=m+1+d.
\end{gathered}                                            \tag{6.1}\label{6.1}
\]
Then \(\Ant_{[2]\times[m]\times[m+1]}=Q_+Q_-\).

For \(0\leq r\leq m\), put \(f_r=[x^r]F\) and \(g_r=[x^r]G\).  The
closed form gives the exact coefficients
\[
 f_r=\binom mr\binom{m+1}r,\qquad g_0=0,
 \qquad
 g_r=\frac{\binom{m-1}{r-1}\binom m{r-1}}
            {\binom{r+1}{2}}\quad(1\leq r\leq m).        \tag{6.2}\label{6.2}
\]
In particular,
\[
 \frac{g_r}{f_r}=\frac{2r}{m(m+1)(r+1)}
 \quad(0\leq r\leq m),
 \qquad
 \frac{f_{m-r}}{f_r}=\frac{m-r+1}{r+1}.                 \tag{6.3}\label{6.3}
\]
The first ratio at \(r=0\) is interpreted using \(g_0=0\).

Since \(d^2=m(m+2)\), equations \eqref{6.2}--\eqref{6.3} imply the two
coefficient identities
\[
\begin{aligned}
 f_{m-r}
  &=(m+1)f_r-\frac{m+1}{2}d^2g_r,\\
 \frac{m+1}{2}g_{m-r}
  &=f_r-\frac{(m+1)^2}{2}g_r
 \qquad(0\leq r\leq m).
\end{aligned}                                             \tag{6.4}\label{6.4}
\]
To verify the first identity in \eqref{6.4}, divide its right-hand side by
\(f_r>0\) to obtain
\[
 (m+1)-\frac{(m+1)m(m+2)}2
       \frac{2r}{m(m+1)(r+1)}
 =\frac{m-r+1}{r+1}.
\]
For \(0<r<m\), after division by \(f_r>0\), both sides of the second
identity similarly reduce to \((m-r)/(m(r+1))\).  The endpoints \(r=0,m\)
follow directly from \(f_0=1\), \(f_m=m+1\), \(g_0=0\), and
\(g_m=2/(m+1)\).

Since \((m+1+d)(m+1-d)=1\), we have
\(\rho^{-1}=m+1-d\).  Using \eqref{6.4} coefficientwise,
\[
\begin{aligned}
 \rho(f_r-sg_r)
 &=\left((m+1)f_r-\frac{m+1}{2}d^2g_r\right)
   +d\left(f_r-\frac{(m+1)^2}{2}g_r\right)\\
 &=f_{m-r}+sg_{m-r},\\
 \rho^{-1}(f_r+sg_r)
 &=\left((m+1)f_r-\frac{m+1}{2}d^2g_r\right)
   -d\left(f_r-\frac{(m+1)^2}{2}g_r\right)\\
 &=f_{m-r}-sg_{m-r}.
\end{aligned}
\]
We have therefore proved the reciprocal factor identities
\[
 \boxed{\quad
 x^mQ_+(1/x)=\rho Q_-(x),
 \qquad
 x^mQ_-(1/x)=\rho^{-1}Q_+(x).
 \quad}                                                   \tag{6.5}\label{6.5}
\]
The leading coefficients of \(Q_+\) and \(Q_-\) are, respectively,
\[
 m+1+d=\rho,\qquad m+1-d=\rho^{-1}.                      \tag{6.6}\label{6.6}
\]
Thus their product has exact degree \(2m\), leading coefficient \(1\),
and constant coefficient \(1\).  Multiplying \eqref{6.5} gives
\[
 x^{2m}\Ant_{[2]\times[m]\times[m+1]}(1/x)
 =\Ant_{[2]\times[m]\times[m+1]}(x),                   \tag{6.7}\label{6.7}
\]
so the polynomial is palindromic.

Write \(A(x)=\Ant_{[2]\times[m]\times[m+1]}(x)\).  By
Theorem~\ref{thm:simple-negative}, its \(2m\) zeros are simple and strictly
negative.  Equation \eqref{6.7} pairs every zero with its reciprocal.  The
fixed point \(-1\) cannot be a zero: differentiating \eqref{6.7} gives
\[
 A'(x)=2m x^{2m-1}A(1/x)-x^{2m-2}A'(1/x).
\]
If \(A(-1)=0\), evaluation at \(x=-1\) would give
\(A'(-1)=-A'(-1)\), hence \(A'(-1)=0\), contrary to simplicity.
Consequently the zeros can be labeled
\[
 -\lambda_1,-\lambda_1^{-1},\ldots,
 -\lambda_m,-\lambda_m^{-1},
 \qquad \lambda_j>1.
\]
Since \(A\) is monic by \eqref{6.6}, it follows that
\[
\begin{aligned}
 A(x)
 &=\prod_{j=1}^m(x+\lambda_j)(x+\lambda_j^{-1})\\
 &=\prod_{j=1}^m\bigl((1+x)^2+\delta_jx\bigr),
 \qquad
 \delta_j=\lambda_j+\lambda_j^{-1}-2
          =\frac{(\lambda_j-1)^2}{\lambda_j}>0.
\end{aligned}
\]
Expanding the product in elementary symmetric functions gives
\[
 \Ant_{[2]\times[m]\times[m+1]}(x)
 =\sum_{i=0}^m
 e_i(\delta_1,\ldots,\delta_m)x^i(1+x)^{2m-2i}.          \tag{6.8}\label{6.8}
\]
Thus \(\gamma_i=e_i(\delta_1,\ldots,\delta_m)>0\) for
\(0\leq i\leq m\).  This proves Theorem~\ref{thm:gamma-intro} and
Conjecture~4.5 of Ding and Dong.

\section{Further questions}

Further questions include a combinatorial interpretation of \eqref{6.8} and
a classification of boxes with all three sides at least \(3\) whose antichain
polynomials are real-rooted.

\medskip
\noindent\textbf{Data availability.}
The reproducibility package is archived at
\href{https://doi.org/10.5281/zenodo.22082487}
{\nolinkurl{doi:10.5281/zenodo.22082487}}.

\medskip
\noindent\textbf{Declaration of generative AI and AI-assisted technologies
in the manuscript preparation process.}
During preparation of this work, the author used OpenAI Codex for literature
searches, exploratory calculations, verification-code development, and
language editing. Its outputs were not used as mathematical evidence. The
author checked the sources, arguments, and computations and takes full
responsibility for the manuscript.

\enlargethispage{3\baselineskip}

\end{document}